\documentclass[10pt, english]{amsart}

\usepackage{amsmath,amssymb,enumerate}

\usepackage[T1]{fontenc}
\usepackage[all]{xy}

\usepackage{babel}
\usepackage{amstext}
\usepackage{amsmath}
\usepackage{amsfonts}
\usepackage{latexsym}
\usepackage{ifthen}

\usepackage{xypic}
\xyoption{all}
\newcommand{\cal}{\mathcal}

\newcommand{\rk}{{\rm rk}}

\newcommand{\lra}{\longrightarrow}

\newcommand{\sJ}{{\mathcal J}}

\newtheorem{lemma1}{}[section]

\newenvironment{lemma}{\begin{lemma1}{\bf Lemma.}}{\end{lemma1}}

\newenvironment{theorem}{\begin{lemma1}{\bf Theorem.}}{\end{lemma1}}
\newenvironment{proposition}{\begin{lemma1}{\bf Proposition.}}{\end{lemma1}}
\newenvironment{corollary}{\begin{lemma1}{\bf Corollary.}}{\end{lemma1}}

\newenvironment{remarks}{\begin{lemma1}{\bf Remarks.}\rm}{\end{lemma1}}

\newenvironment{conjecture}{\begin {lemma1}{\bf Conjecture.}}{\end{lemma1}}

\newenvironment{remark*}{{\bf Remark.}}{}
\newenvironment{example*}{{\bf Example.}}{}

\newcommand{\R}{\ensuremath{\mathbb{R}}}
\newcommand{\Q}{\ensuremath{\mathbb{Q}}}
\newcommand{\Z}{\ensuremath{\mathbb{Z}}}
\newcommand{\C}{\ensuremath{\mathbb{C}}}
\newcommand{\N}{\ensuremath{\mathbb{N}}}

\newcommand{\merom}[3]{\ensuremath{#1:#2 \dashrightarrow #3}}

\newcommand{\holom}[3]{\ensuremath{#1:#2  \rightarrow #3}}

\makeatletter
\ifnum\@ptsize=0  \fi \ifnum\@ptsize=2  \fi 

\newcommand\sF{{\mathcal F}}

\newcommand\sI{{\mathcal I}}

\newcommand\sL{{\mathcal L}}
\newcommand\sO{{\mathcal O}}

\newcommand\sS{{\mathcal S}}

\newcommand\bR{{\mathbb R}}

\newcommand\bC{{\mathbb C}}

\newcommand\bN{{\mathbb N}}

\DeclareMathOperator*{\supp}{Supp}

\title{Uniformisation in dimension four: \\
towards a conjecture of Iitaka} 
\date{March 28, 2011}

\subjclass[2000]{32Q30, 32J27, 14E30, 14J35}
\keywords{universal cover, MMP}

\author{Andreas H\"oring}
\author{Thomas Peternell}
\author{Ivo Radloff}

\address{Andreas H\"oring, Universit\'e Pierre et Marie Curie and Albert-Ludwig Universit\"at Freiburg} 
\curraddr{Mathematisches Institut, Albert-Ludwigs-Universit\"at
  Freiburg, Eckerstra{\ss}e 1, 79104 Freiburg im Breisgau, Germany}
\email{hoering@math.jussieu.fr}

\address{Thomas Peternell, Mathematisches Institut, Universit\"at Bayreuth, 95440 Bayreuth, 
Germany 
}
\email{thomas.peternell@uni-bayreuth.de}

\address{Ivo Radloff, Mathematisches Institut, Universit\"at Bayreuth, 95440 Bayreuth, 
Germany 
}
\email{ivo.radloff@uni-bayreuth.de}

\begin{document}

\begin{abstract} 
Let $X$ be a compact K\"ahler manifold whose universal covering is $\C^n$.
A conjecture of Iitaka claims that $X$ is a torus,  up to finite \'etale cover.
We prove this conjecture in various cases in dimension four. 
We also show that in the projective case Iitaka's conjecture
is a consequence of the non-vanishing conjecture.
\end{abstract}

\maketitle

\section{Introduction}

The following conjecture was posed by Iitaka:

\begin{conjecture} {\rm ($I_n$)}
Let $X$ be a compact K\"ahler manifold of dimension $n$ such that 
the universal covering $\tilde X$ is biholomorphic to $\C^n$.
Then $X$ is a torus, up to finite \'etale cover.
\end{conjecture}

Nakayama proved the conjecture for projective manifolds of dimension at most three and more generally for K\"ahler manifolds 
of Kodaira dimension $\kappa(X) \ge n-3$ or irregularity $q(X) \ge n-3$ (\cite{Nak99}, \cite{Nak99b}). 
His methods do not work in general in dimension four (and higher) without assuming the abundance conjecture.
We will mainly be concerned with this dimension. Our main results can be summarized as follows.

\begin{theorem} \label{maintheorem}
Let $X$ be a smooth compact K\"ahler fourfold with universal cover $\tilde X \simeq \C^4.$ 
Then $X$ is a torus (up to finite \'etale cover) if one of the following conditions holds.
\begin{enumerate}
\item $X$ is projective with Kodaira dimension $\kappa (X) \geq 0$;
\item $X$ is not projective, but covered by positive-dimensional compact subvarieties;
\item $X$ does not contain surfaces or divisors;
\item $\kappa(X) \geq 0$ and $K_X^2=0$.
\end{enumerate} 
\end{theorem} 

The theorem is a summary of the Corollaries \ref{corollaryfoura}, \ref{corollaryfourb},  \ref{corollaryfourc},
Theorem \ref{theoremfourd}, and the Corollary \ref{corollaryfoure}.
The theorem establishes Iitaka's conjecture in dimension four except the following two cases.

- First, there is the potential case that $X$ is projective with $K_X$
nef and $\kappa (X) = - \infty.$ This case should not exist: a projective
manifold with $\kappa (X) = - \infty$ should be uniruled. This is however
only known in dimension at most three. 

- Second, $X$ could not be covered by positive-dimensional compact subvarieties,
but contain some surfaces or divisors. Again this case should not exist:
since we expect $X$ to be a torus  (up to finite \'etale cover), it should either
have a covering family of compact subvarieties or have no subvarieties at all. 

It is important to note that for a smooth projective $n$-fold $X$ with universal cover $\tilde X \simeq \C^n$ 
the canonical bundle $K_X$ is automatically nef; otherwise
by the cone theorem, $X$ would contain a rational curve, which of course is not possible since $\tilde X$ is Stein. 
Now the abundance conjecture predicts that some multiple of $K_X$ is spanned, defining a holomorphic Iitaka fibration $f: X \lra Y$. 
Since $X$ has generically large fundamental group, a result of Koll\'ar \cite[6.3. Thm.]{Kol93} 
says that - possibly after finite 
\'etale cover - 
$f$ is a smooth abelian group scheme over $Y$ where the base $Y$ has general type. However, a classical theorem of Kobayashi and Ochiai \cite{KO75}
says there is no non degenerate map from $\C^n$ to a (positive dimensional) manifold of general type. Hence $Y$ must be a point and 
$X$ is abelian as desired. 

In summary the abundance conjecture implies Iitaka's conjecture. 

The abundance conjecture is known to hold true for projective manifolds of dimension at most three
\cite{Kol92}, but is essentially open in higher dimensions.
We are able to show that Conjecture $I_n$ is actually a consequence of a weaker conjecture, the so-called non-vanishing conjecture:

\begin{conjecture} (Non-vanishing for klt pairs) 
Let $X$ be a $n$-dimensional normal projective variety and $\Delta$ an effective $\Q$-divisor
such that $(X, \Delta)$ is klt and $K_X+\Delta$ is nef. Then we have $\kappa(K_X+\Delta) \geq 0$.
\end{conjecture}

Our result is 

\begin{theorem} 
Let $X$ be a projective manifold of dimension $n$ whose universal cover is $\C^n$.
Suppose that $\kappa(X) \geq 0$ and that the non-vanishing conjecture
holds in dimension $n-1$. Then $X$ is an \'etale quotient of a torus.
\end{theorem}

The statement is based on a recent extension result of Demailly, Hacon and  P{\v a}un \cite{DHP10}.
Actually we prove

\begin{theorem} Let $X$ be a projective manifold of dimension $n$ without rational curves. Suppose that 
$\kappa (X) \geq 0$ and that the non-vanishing conjecture holds in dimension $n-1$. Then $K_X$ is semi-ample. 
\end{theorem} 

If $X$ is a non-algebraic compact K\"ahler manifold we encounter
a number of additional difficulties.
Since we do not have a cone theorem it is not clear at all if the canonical bundle $K_X$ is pseudoeffective
(or even nef). However we are able to show that in dimension four, $K_X$ is at least
pseudo-effective, i.e., its Chern class is represented by a positive closed current. 

This works for the following reason. If $X$ is a compact
K\"ahler manifold with generically large fundamental group, then $\chi(X,\sO_X) = 0$ unless
$X$ is of general type. This is particularly useful in the non-algebraic context. In fact, suppose $\dim X = 4$ and
$q(X) = 0$. A non-algebraic $X$ carries a
holomorphic $2-$form, therefore we obtain at least two independent holomorphic $3-$forms which give foliations on 
$X$ by curves.  Since the foliations 
cannot have rational leaves, a remarkable theorem of Brunella says that $K_X$ must be pseudo-effective. 

If $X$ is a meromorphic fibre space, then it is very natural to proceed by induction on the dimension. 
Therefore we need to amplify Conjecture $I_n$ to adapt it to induction procedures. This is done in Section 2. Natural fibre spaces occurring are the algebraic reduction and fibre spaces arising by forming 
quotients defined by covering families of subvarieties. The general theorem in Section 2 says that we basically can reduce our studies to projective manifolds and simple manifolds, i.e. manifolds which are 
not covered by positive-dimensional subvarieties.

The case of simple fourfolds is quite difficult. In view of the expected result it is
very natural to require $X$ to have ``not too many compact subvarieties'': a simple torus does not admit any subvariety at all. If we assume our fourfold $X$ not to have divisors and surfaces, then, arguing by contradiction, 
we can use the $2$- and $3$-forms mentioned above to construct
a unitary flat vector bundle on $X$, giving rise to a unitary representation of $\pi_1(X).$ In that situation a theorem of Mok gives either a map to a torus or a meromorphic map to a variety of general type, 
contradicting the simplicity of $X$. 

If $\kappa (X) \geq 0,$ we may study a divisor $D \in \vert mK_X \vert.$ 
At least when $K_X^2 = 0,$ or more generally, when $K_X^2 \cdot \omega^2 \leq 0$ for some K\"ahler form $\omega,$ we completely 
describe the structure of $D$ and prove that $K_X \equiv 0.$ 
A disadvantage of this approach is that we still have to assume the existence of a global section
for some $m K_X$.

One way to construct holomorphic objects is via the hard Lefschetz theorem. Let $X$ again be a fourfold with $\tilde X \simeq \C^4$
and assume $K_X \cdot c_2(X) \ne 0.$ Then by Riemann-Roch, some $h^q(X,mK_X)$ grows at least quadratically. If now $K_X$ is
hermitian semi-positive, or more generally, if $K_X$ has a possibly singular metric with positive curvature current such 
that the multiplier ideal sheaves $\sJ(h^m)$ are trivial for large $m$, then also $h^0(X,\Omega^q_X \otimes mK_X)$ grows
quadratically, so that $X$ has a non-constant meromorphic function. 
Hence $X$ is projective, and then we prove in Corollary  \ref{corollarymultiplier} that $\kappa (X) \geq 0$. In total we may state

\begin{theorem} Let $X$ be a compact K\"ahler fourfold with universal covering $\tilde X \simeq \C^4.$ If $K_X$ is hermitian
semi-positive and if $K_X^2 \cdot c_2(X) \ne 0,$ then $\kappa (X) = 0.$ 
\end{theorem}

\begin{corollary} Let $X$ be a smooth projective fourfold with universal covering $\tilde X \simeq \C^4.$ If $K_X$ is
hermitian semi-positive, then $K_X^2 \cdot c_2(X) = 0$ and therefore $\chi(X, mK_X) = 0$ for all integers $m.$  
\end{corollary} 

For the proofs we refer to Corollary \ref{corollarymultiplier} and Corollary \ref{corollaryfourf}. 

{\bf Acknowledgements.} 
We thank the Forschergruppe 790 ``Classification of algebraic surfaces and compact complex manifolds'' of the Deutsche
Forschungsgemeinschaft for
financial support. 
A.H\"oring was partially also supported by the A.N.R. project ``CLASS''.

\begin{center}
{\bf Notation}
\end{center}

Let $X$ be a  compact K\"ahler manifold, and $\tilde X$ its universal cover. 
We say, following \cite{Kol93}, that $X$ has large fundamental group if for every positive-dimensional 
subvariety $Z \subset X$,
with normalisation $\bar Z \rightarrow Z$, 
the morphism $\pi_1(\bar Z) \rightarrow \pi_1(X)$ has infinite image. In particular $\tilde X$
has no compact subvarieties. The fundamental group is generically large if the preceding
property holds for every subvariety passing through a very general point.

Given a compact K\"ahler manifold $X$ we denote by $\kappa(X)$ its Kodaira dimension,
$a(X)$ its algebraic dimension, $q(X):=h^1(X, \sO_X)$ its irregularity and by 
\[
\tilde q(X) := \mbox{max} \{ \
q(X') \ | \ X' \rightarrow X \ \mbox{finite \'etale}
\}.
\]
We say that $X$ is simple if there is no positive-dimensional subvariety through a
very general point $x \in X$; this is equivalent to supposing that $X$ has no covering
family of positive-dimensional subvarieties. \\
The structure of simple K\"ahler manifolds is quite mysterious, even in dimension 3. It is expected
that a simple K\"ahler threefold is bimeromorphic to a quotient $T/G$ of a torus $T$ by a finite group $G$.
In higher dimensions, basically the only known example - up to quotients and bimeromorphic transformations - 
are general tori and ``general'' Hyperk\"ahler manifolds.

\section{The projective case}

The aim of this section is to show that for projective manifolds Conjecture $I_n$
is a consequence of the following non-vanishing conjecture

\begin{conjecture} (Non-vanishing for klt pairs) \label{conjecturenonvanishing}
Let $X$ be a $n$-dimensional normal projective variety and $\Delta$ an effective $\Q$-divisor
such that $(X, \Delta)$ is klt and $K_X+\Delta$ is nef. Then $\kappa(K_X+\Delta) \geq 0$.
\end{conjecture}

Then we may state

\begin{theorem} \label{theoremprojectivecase}
Let $X$ be a projective manifold of dimension $n$ such that the universal cover is $\C^n$.
Suppose that $\kappa(X) \geq 0$ and the non-vanishing Conjecture \ref{conjecturenonvanishing}
holds in dimension $n-1$. Then $X$ is an \'etale quotient of a torus.
\end{theorem}

This follows immediately from the following more general Theorem \ref{theoremsemiample} and \cite{Nak99}.

\begin{theorem} \label{theoremsemiample}
Let $X$ be a projective manifold of dimension $n$ without rational curves.
Suppose that $\kappa(X) \geq 0$ and the Conjecture \ref{conjecturenonvanishing}
holds in dimension $n-1$. Then $K_X$ is semiample.
\end{theorem}

Since the non-vanishing conjecture holds in dimension at most three, cf. \cite{Kol92},
we deduce from Theorem \ref{theoremprojectivecase} 

\begin{corollary} \label{corollaryfoura}
Let $X$ be a projective fourfold such that $\kappa(X) \geq 0$.
Then $I_4$ holds for $X$.
\end{corollary}

The proof of Theorem \ref{theoremsemiample} is prepared by two lemmata; here we use freely notations from the
Minimal Model Program, abbreviated as usual by MMP. We refer e.g. to \cite{KM98} for the basic definitions. 

\begin{lemma} \label{lemmatermination}
Let $A$ be a $n$-dimensional normal projective variety without rational curves, and
let $\holom{\mu}{X}{A}$ be a birational morphism from a normal projective $\Q$-factorial variety $X$.
Let $S \subset X$ be an irreducible divisor that is not $\mu$-exceptional
and $\Delta$ an effective $\Q$-divisor on $X$ such that $\lfloor \Delta \rfloor=0$ 
and $(X, S+\Delta)$ is plt. 

Then $(X, S+\Delta)$ has a minimal model.
\end{lemma}

The proof consists combination of a combination of \cite{BCHM10} and special termination arguments \cite[Thm.4.2.1]{Fuj07};
we give full details for the convenience of the reader.

\begin{proof}
Note that since $A$ has no rational curves, it is sufficient to construct a relative minimal model over $A$.
Let now $H$ be a sufficiently ample Cartier divisor such that $(X, S+\Delta+H)$ is plt and 
$K_X+S+\Delta+H$ is nef. We start to run a MMP on $(X, S+\Delta)$ over $A$ with scaling by $H$.
This induces a MMP on $(S, \Delta_S)$ over $\mu(S)$ with scaling by $H|_S$.

Since $(X, S+\Delta)$ is plt, the pair $(S, \Delta_S)$ is klt. Since $S \rightarrow \mu(S)$ is birational the 
MMP with scaling by $H|_S$ terminates after finitely many steps \cite[Cor.1.4.2]{BCHM10}, i.e.
there exists a birational map
$$
\merom{\psi}{X}{X'}
$$
that extracts no divisor such that if we set $S':=\psi_* S$ and $\Delta':=\psi_* \Delta$, then $(X', S'+\Delta')$ is plt and $(S', \Delta'_{S'})$
is a minimal model. We claim that if \holom{\varphi}{X'}{Y} is an extremal contraction
for the pair $(X', S'+\Delta')$, then the exceptional locus is disjoint from $S'$.
Assuming this for the time being, let us see how to conclude: 
since the exceptional loci of the $(X', S'+\Delta')$-MMP are disjoint from $S'$ (this property does not change under the MMP), it is also a MMP for $(X', \Delta')$. Since $X'$ is $\Q$-factorial, the pair
$(X', \Delta')$ is plt without integral boundary divisor, so klt. Thus the MMP with scaling terminates
by \cite[Cor.1.4.2]{BCHM10}.

{\em Proof of the claim.} Set $T$ for the image $\varphi(S')$. We claim that the morphism 
$\holom{\varphi|_{S'}}{S'}{T}$
is finite: if $C$ is contracted by $S' \rightarrow T$, it is contained in $S'$ and contracted by $\varphi$.
Thus one has
$$
(K_{S'}+\Delta'_{S'}) \cdot C = (K_{X'}+S'+\Delta') \cdot C <0,
$$
contradicting the nefness of $K_{S'}+\Delta'_{S'}$.

Assume without loss of generality that the contraction $\varphi$ is small
and denote by \holom{\varphi^+}{X^+}{Y} the flip of $\varphi$. Set $S^+ \subset X^+$ for the strict transform of $S'$.  
We argue by contradiction and 
suppose that the $\varphi$-exceptional locus meets $S'$. Then there exists a curve $C \subset X'$
that is contracted by $\varphi$ and such that $S'\cdot C>0$. Thus for every curve $C^+ \subset X^+$
that is contracted by $\varphi^+$ one has $S^+\cdot C^+<0$, i.e. the exceptional locus of $\varphi^+$ is contained
in $S^+$. In particular the map $\holom{\varphi^+|_{S^+}}{S^+}{T}$ is not finite.
Yet it is a well-known property of the flip that $S^+$ is the relative canonical model of $S' \rightarrow T$.
Since $S' \rightarrow T$ is relatively nef, we have a morphism $g: S' \rightarrow S^+$ such that
$\varphi^+|_{S^+} \circ g = \varphi|_{S'}$ \cite[2.22 Thm.]{Kol92}. Since $\varphi|_{S'}$ is finite, but $\varphi^+|_{S^+}$ is not, 
we have reached a contradiction.
\end{proof}
 
\begin{lemma} \label{lemmaminimalmodel}
Let $A$ be a $n$-dimensional normal projective variety without rational curves,
let $\holom{\mu}{X}{A}$ be a birational morphism from a projective manifold $X$.
Suppose that $\kappa(X)=0$ and that Conjecture \ref{conjecturenonvanishing}
holds in dimension $\leq n-1$. Then $X$ has a good minimal model, i.e.
there exists a birational map \holom{\psi}{X}{X'} such that 
$X'$ is klt and $K_{X'} \sim_\Q 0$.
\end{lemma}

\begin{proof}
Since $A$ has no rational curves, a relative minimal model for $X$ over $A$ is a minimal model for $X$ itself. 
Thus we know by \cite[Cor.1.4.2]{BCHM10} that there exists 
a birational map \holom{\psi}{X}{X_{\min}} such that $X_{\min}$ is klt and $K_{X_{\min}}$ is nef.
We argue by contradiction and suppose that $K_{X_{\min}} \not\sim_\Q 0$.
Since $\kappa(X)=\kappa(X_{\min})=0$ there exists an effective $D_{\min} \in |m K_{X_{\min}}|$ for some $m$.
The natural morphism $\mu':X_{\min} \rightarrow A$ is birational, so the negativity lemma (\cite[Lemma 3.6.2]{BCHM10}) implies
that there exists an irreducible component $S' \subset D_{\min}$ that $\mu'|_{S'}$ is birational.
In particular $S'$ is not uniruled.
Thus if $D \in |m K_X|$ is the divisor such that $\psi_* D=D_{\min}$, then $D$ has an irreducible 
component that is not uniruled.

Let $\holom{\mu}{X'}{X}$ be a log-resolution, i.e. the union of the exceptional locus 
and the support of the pull-back $\mu^* D$ form a normal-crossing divisor.
Since $K_{X'} \simeq \mu^* K_X+E$ with $E$ effective, we have
an isomorphism 
$$
K_{X'}+\mu^* D+mE \simeq  (m+1) K_{X'},
$$
in particular $\kappa(X', \mu^* D+mE) = \kappa(X')=\kappa(X)$.
Let now $S$ be the strict transform of an irreducible component of $D$ that is not uniruled,
and set $\Delta:=\varepsilon (\supp (\mu^* D+mE)-S)$ with $0< \varepsilon \ll 1$.
Then the pair $(X', S+\Delta)$ is plt and satisfies
$$
\kappa(X', S+\Delta) \leq \kappa(X', \mu^* D+mE)
$$
Moreover 
$$
K_{X'}+S+\Delta \sim_\Q \mu^* D+E+S+\Delta=:D' 
$$
has the property
$$
S  \subset \supp(D') \subset \supp(S+\Delta).
$$
Since $\mu|_S$ is birational onto its image, we know by Lemma \ref{lemmatermination} 
that $(X', S+\Delta)$ has a minimal model.
Since the divisor $S$ is not uniruled, it is not contracted by any MMP.
Thus up to replacing $(X', S+\Delta)$ by its minimal model (which does not change the preceding
properties) we can suppose without loss of generality that $K_{X'}+S+\Delta$ is nef.
In particular the pair $(S, \Delta_S)$ is klt and a minimal model. Thus by
Conjecture \ref{conjecturenonvanishing} in dimension $n-1$, one has
$$
\kappa(S, \Delta_S) \geq 0.
$$
By \cite[Cor.1.8]{DHP10} this shows that $\kappa(X', S+\Delta) \geq 1$, a contradiction.
\end{proof}

\begin{proof} of Theorem \ref{theoremsemiample}
Since $X$ has no rational curves, it is a minimal model.

If $\kappa(X)=0$ we know by Lemma \ref{lemmaminimalmodel} that $X$ has a good minimal model.
Thus by \cite[Prop.2.4]{Lai10} the minimal model $X$ is also good.

If $\kappa(X) \geq 1$, we denote by \holom{\mu}{X'}{X} a resolution of the indeterminacies 
of some pluricanonical system $|m K_X|$ for $m \gg 0$ sufficiently divisible.
Let $F$ be a general fibre of the Iitaka fibration on $X'$, then $\kappa(F)=0$
and the normalisation of $\mu(F)$ does not contain any rational curves. Thus
$F$ has a good minimal model by Lemma \ref{lemmaminimalmodel}.
By \cite[Thm.0.2, Prop.2.4]{Lai10} this implies that $X$ is a good minimal model.
\end{proof}


\section{A reduction step}

In this section we show that in order to prove Conjecture $I_n$ for
non-algebraic compact K\"ahler manifolds that are not simple, it is sufficient
to prove (generalised versions of) $I_k$ for $k<n$.  

\begin{conjecture} {\rm ($T_n$)} \label{conjecturetn}
Let $X$ be a compact K\"ahler manifold of dimension $n$ with generically large fundamental group. Suppose furthermore one of the following:
\begin{enumerate}
\item There exists a proper modification $\tilde \C^k \rightarrow \C^k$ and a surjective 
map $\tilde \C^k \rightarrow \tilde X$\footnote{We do not suppose that $k=n$, this additional flexibility is useful
for the induction argument in the proofs below.};
\item $\kappa(X) \leq 0$.
\end{enumerate}
Then there exists a finite \'etale cover $X' \rightarrow X$ such that $X'$ is bimeromorphic to a torus.
\end{conjecture}

A more general conjecture has been made by \cite{CZ05}:

\begin{conjecture} {\rm($CZ_n$)} \label{conjectureczn}
Let $X$ be a compact K\"ahler manifold of dimension $n$ 
with generically large fundamental group.
Then $X$ is (up to finite \'etale cover) bimeromorphic to a torus submersion over
a variety of general type.
\end{conjecture}

\begin{remarks}
\begin{enumerate}
\item  It is clear that the conjecture $T_n$ implies Iitaka's conjecture $I_n$.
\item The conjecture $CZ_n$ holds for $n \leq 3$ by \cite{CZ05}, thus also $T_n$ holds for $n \leq 3.$
\item By the Kobayashi-Ochiai theorem \cite{KO75} 
a manifold satisfying the conditions in $T_n$
does not admit a morphism onto a variety of general type.
Note moreover that $T_n$ is a bimeromorphic property, so one can always replace $X$ by some bimeromorphic
model. 
\end{enumerate}
\end{remarks}

\begin{lemma} \label{lemmainduction}
Let $X$ be a compact K\"ahler manifold of dimension $n$
satisfying the conditions of $T_n$. Suppose that the conjecture $T_k$ holds for every $k<n$.

Let $\merom{f}{X}{Y}$ be a meromorphic fibration 
such that the general fibre $X_y$ is not of general type. Then $T_n$ holds for $X$.
\end{lemma}

\begin{proof}
Since $T_n$ is a bimeromorphic property we can suppose that $f$ is holomorphic.
Moreover we can suppose that $\kappa(X_y) \leq 0$: otherwise we replace $f$ by the
relative Iitaka fibration. Since $X$ has generically large fundamental group, 
this also holds for $X_y$.
Thus $X_y$ satisfies the conditions of $T_{\dim X_y}$, so there exists a finite \'etale cover
$X_y' \rightarrow X_y$ such that $X_y'$ is bimeromorphic to a torus.
By \cite[Thm.8.6]{Nak99b} (cf. \cite[Thm.6.3]{Kol93} for the projective case) this implies that
there exists a finite \'etale cover $X' \rightarrow X$ such that $X'$ is 
bimeromorphic to a smooth torus fibration over a base of dimension $\dim Y$.
Since $T_n$ is a bimeromorphic property and invariant under finite \'etale cover, 
this shows that we can suppose that the fibration $f$ is a smooth torus fibration.
Since $X$ has generically large fundamental group, 
this also holds for $Y$ by  \cite[Cor.8.7]{Nak99b}.

Now note the following: if $X$ satisfies the property 1) (resp. property 2)) in $T_n$,
then $Y$ also satisfies the property  1) (resp. property 2)) in $T_{\dim Y}$.
Since $Y$ has generically large fundamental group we know by $T_{\dim Y}$
that there exists a finite \'etale cover $Y' \rightarrow Y$ such that $Y'$ is bimeromorphic to a torus $A$.
Hence (up to replacing $X$ by the fibre product $X \times_Y Y'$) we can suppose that $Y$ has
a bimeromorphic map $ Y \rightarrow A$ to a torus $A$.
In particular one has $\pi_1(Y) \simeq \Z^{2 \dim Y}$. 
Since $f$ is submersive we have an injection
$\pi_1(F) \hookrightarrow \pi_1(X)$ \cite[Lemme 2.1]{Cla10}
and an exact sequence
$$
0 \rightarrow \pi_1(F) \rightarrow \pi_1(X) \rightarrow \pi_1(Y) \rightarrow 0.
$$
Thus $\pi_1(X)$ is an extension of abelian groups, hence almost abelian by \cite{Cam98b}.
In particular a finite \'etale cover of $X$ has free fundamental group, this easily implies that it
is birational to its Albanese torus. 
\end{proof}

The following result shows that it is sufficient to prove the conjecture $T_n$ for
projective manifolds and manifolds with algebraic dimension zero.

\begin{proposition} \label{propositionmiddlecases}  
Let $X$ be a compact K\"ahler manifold of dimension $n$
such that $0<a(X)<n$ satisfying the conditions of $T_n$.
Suppose that the conjecture $T_k$ holds for every $k<n$.
Then $T_n$ holds for $X$.
\end{proposition}

\begin{proof}
Let \merom{f}{X}{Y} be the algebraic reduction of $X$, then the general fibre $X_y$ 
satisfies $\kappa(X_y) \leq 0$. Conclude with Lemma \ref{lemmainduction}.
\end{proof}

\begin{corollary} \label{corollaryfourb}
$T_4$ holds for compact K\"ahler fourfolds $X$ with $0<a(X)<4$.
\end{corollary}

\begin{proposition} \label{propositionnotsimple}  
Let $X$ be a compact K\"ahler manifold of dimension $n$
and algebraic dimension $a(X)=0$ satisfying the conditions of $T_n$.
Suppose that the conjecture $CZ_k$ holds for every $k \leq n-2$.
If $X$ is not simple, then $T_n$ holds for $X$.
\end{proposition}

\begin{proof} 
Let $(Z_t)_{t \in T} \subset X$ be a covering family of positive-dimensional subvarieties,
and denote by \merom{\varphi}{X}{Y} the corresponding quotient fibration
\cite{Cam04a}, i.e. $\varphi$ is almost holomorphic and a very general fibre $F$ is an equivalence class for the relation defined by chains of $Z_t$'s.
Note that a compact manifold with $a(X)=0$ contains only finitely many divisors \cite{FF79}.
In particular the varieties $Z_t$ and $F$ have dimension at most $n-2$.

{\em 1st case. $\dim Y>0$.} By Lemma \ref{lemmainduction} it is sufficient to show
that a general fibre $F$ is not of general type. Suppose to the contrary that $F$ is of general type. Then by 
\cite[Cor.3,p.412]{Cam85b},\cite{Fuj83} we know that since $a(X) = a(Y)$ and since $F$ is projective,  $F$ is almost homogenous.
In particular $-K_F$ is effective, so $F$ is far from being of general type, contradiction.

{\em 2nd case. $\dim Y=0$.} Note first that in this case $Z_t$ is not algebraic:
otherwise two generic points are connected by a chain of curves, so $X$ is projective by
a theorem of Campana \cite[Cor., p.212]{Cam81}.

Let $Z_t'$ be a desingularisation of a general
member of the family. Since the fundamental group of $Z_t'$ is generically large,
we know by $CZ_k$ that (up to finite \'etale cover) the manifold $Z_t'$ is bimeromorphic
to a torus submersion over a variety of general type. 
Up to replacing the family $Z_t$
by the covering family given by the images of the tori, we can suppose that $Z_t'$
is bimeromorphic to a torus. In particular the fundamental group of the $Z_t'$ is abelian.
Since two general points of $X$ are connected by chains of $Z_t$'s, we know
by \cite{Cam98b} that $\pi_1(X)$ is almost abelian. This immediately implies the statement.
\end{proof}

\begin{corollary} \label{corollaryfourc}
Let $X$ be a compact K\"ahler manifold of dimension $n \leq 5$ such that $a(X)=0$,
but $X$ is not simple. Then $T_n$ holds for $X$.
\end{corollary}

\section{Simple compact K\"ahler fourfolds}

\subsection{Fourfolds without surfaces and divisors}
\label{subsectionverynonalgebraic}

We start with some technical preparation:

\begin{proposition} \label{propositionpseudoeffective}
Let $X$ be a compact K\"ahler fourfold with generically large fundamental group. Then $K_X$ is pseudo-effective.
\end{proposition} 

\begin{proof} By \cite{BDPP04} this is obvious if $X$ is projective (because $X$ is uniruled if $K_X$ is not pseudo-effective); 
so suppose that this is not the case.
By \cite{Cam95} we thus have $\chi(X,\sO_X) = 0.$ If $q(X) > 0,$ we consider the Albanese map $\alpha: X \to A$ with image $Y$ of dimension $d.$ 
The general fibre $X_y$ has generically large fundamental group, hence $\kappa (X_y) \geq 0$ by \cite{CZ05}. 
If $\dim Y=1$ or $3$ we conclude by $C_{n,m}$ \cite{Fuj78, Uen87}.
Since $C_{4,2}$ seems not to be known in the K\"ahler case, we have to use a slightly
different argument: if $F$ is a general $\alpha$-fibre, its fundamental group is generically large.
Moreover we can suppose that $F$ is not covered by curves, since otherwise we can factor $\alpha$
birationally through a fibration with relative dimension one, hence \cite{Uen87} applies.
It follows from the classification of surfaces that $F$ is birational to a torus.
Since $\pi_1(X)$ is generically large, \cite[Thm.8.6]{Nak99b} shows that $\alpha$ is bimeromorphically
equivalent to a torus submersion. In this case $C_{4,2}$ is trivial.
\\
So let $q(X) = 0.$ Then $h^0(X,\Omega^3_X) =
h^3(X,\sO_X) > 0,$ so $X$ carries a holomorphic $3-$form which induces a 
rank one foliation $-K_X \to T_X.$ By Brunella \cite{Bru06} this is
a foliation by rational curves unless $K_X$ is pseudo-effective. 
\end{proof} 

The last lines of the proof actually show 

\begin{proposition} 
Let $X$ be a compact K\"ahler manifold of dimension $n$ that is not uniruled. Assume that $h^{n-1}(X,\sO_X) \ne 0.$ Then $K_X$ is pseudo-effective.
\end{proposition} 

\begin{lemma} \label{lquadomega} 
Let $X$ be a compact K\"ahler fourfold without divisors and surfaces. Let $L$ be a pseudo-effective holomorphic line bundle on $X.$
$$ c_1(L)^2 \cdot \omega^2 \geq 0 $$
for all K\"ahler forms $\omega.$ 
\end{lemma} 

\begin{proof} It suffices to show that 
$$ (c_1(L) + \epsilon \omega)^2 \cdot \omega^2 \geq 0 $$
for all $\epsilon > 0.$
By \cite{Dem92, Bou04}, see also \cite[3.1]{BDPP04}, there exists a sequence of blow-ups $\mu: \tilde X \to X$ such that 
$$ \mu^*(c_1(L) + \epsilon \omega) = E + \alpha $$
with an effective $\bR-$divisor $E$ and a K\"ahler class $\alpha.$ 
Since $\dim \mu({\rm supp}E) \leq 1, $ we have
$$ E^2 \cdot \mu^*(\omega)^2 = E \cdot \alpha \cdot \mu^*(\omega)^2  = 0,$$
and therefore 
$$  (c_1(L) + \epsilon \omega)^2 \cdot \omega^2 = (E + \alpha)^2 \cdot \mu^*(\omega)^2 = \alpha^2 \cdot \mu^*(\omega)^2 > 0.$$ 
\end{proof} 

\begin{theorem} \label{theoremfourd}
Let $X$ be a smooth compact K\"ahler fourfold with universal cover $\tilde X \simeq \C^4.$ 
Suppose that $X$ does not contain any surfaces or divisors.
Then $X$ is a torus up to finite \'etale cover.
\end{theorem} 

\begin{proof}
If $q(X)>0$ we conclude by \cite[Thm.8.12]{Nak99b}.
If $H^0(X, mK_X) \not= 0$ for some $m \not= 0$, then we have $K_X \equiv 0$ and the statement follows from the 
Beauville-Bogomolov decomposition theorem. 
We argue by contradiction and suppose that this is not the case, i.e., $H^0(X, mK_X) = 0$ for $m \not = 0$ and $q(X) = 0,$
even after finite \'etale cover, i.e., $\tilde q(X) = 0.$

Since $\chi(X, \sO_X)=0$ and $X$ is not projective, we have $h^0(X, \Omega_X^2)\geq 1$
and therefore $h^0(X, \Omega_X^3)\geq 2$. Since $X$ has no divisors, the subsheaf
$$
\sO_X^{\oplus 2} \rightarrow \Omega_X^3
$$
is saturated. Thus the quotient $Q$ is torsion-free, hence locally free in 
the complement of an analytic subset $Z$ of dimension at most $1$. Moreover its determinant
is isomorphic to $\det \Omega_X^3 \simeq 3 K_X$.
The exact sequence 
$$
 \qquad 0 \rightarrow \sO_X^{\oplus 2} \rightarrow \Omega_X^3 \rightarrow Q \rightarrow 0 \eqno(*)
$$
induces the following sequences (which are exact in the complement of $Z$)
$$
0 \rightarrow \det(K_X \otimes Q^*)  \rightarrow \sF \rightarrow (2 K_X \otimes Q^*)^{\oplus 2} \rightarrow 0
$$
and
$$
0 \rightarrow \sF \rightarrow \Omega_X^2 \rightarrow 2 K_X \rightarrow 0.
$$
Since $Z$ has codimension at least three, the section of any reflexive sheaf on $X \setminus Z$
extends to $X$. Since $H^0(X, \Omega_X^2) \neq 0$ and $H^0(X, 2K_X)=0$, we find $H^0(X, \sF) \not= 0$. Then $H^0(X, \det(K_X \otimes Q^*)) = H^0(X,-K_X) = 0$ implies
$$H^0(X, 2 K_X \otimes Q^*) \neq 0.$$ 
Since $X$ has no divisors, the reflexive sheaf $2 K_X \otimes Q^*$ can 
be written as an extension
$$ 0 \rightarrow \sO_X \rightarrow 2 K_X \otimes Q^* \rightarrow L \rightarrow 0, $$
where $L$ is a torsion-free rank one sheaf such that $L^{**} \simeq K_X$.

{\em 1st case. $Q^*$ is not stable with respect to some K\"ahler form $\omega$.}
In this case there exists a line bundle $M \subset 2 K_X \otimes Q^*$ such that
$$
M \cdot \omega^3 \geq \frac{K_X \cdot \omega^3}{2}. 
$$
If the map $M \rightarrow L$ is not zero, it is an isomorphism since $X$ contains no divisor.
Thus $2 K_X \otimes Q^*$ is locally free and 
$$2 K_X \otimes Q^* \simeq \sO_X \oplus K_X.$$
In particular $K_X \otimes Q^* \simeq \sO_X \oplus K_X^*$ has a global section. 
Yet by  the exact sequence (*), we have an inclusion $K_X \otimes Q^* \subset \Omega^1_X$, 
so $q(X) \geq 1$ which we excluded. \\
Consequently the map $M \rightarrow L$ is zero, hence $M \simeq \sO_X$. 
Thus we get $K_X \cdot \omega^3\le 0$. Since $K_X$ is pseudoeffective by 
Proposition \ref{propositionpseudoeffective}, we obtain $K_X \equiv 0$,
again a contradiction.

{\em 2nd case. $Q^*$ is stable (for all K\"ahler forms $\omega$).}
We have $L \simeq \sI_Z \otimes K_X$ where $Z$ 
is an analytic subspace of codimension at least three. We obtain $c_2(2 K_X \otimes Q^*) = 0$. 
This translates into
$$ c_2(Q) = 2K_X^2.$$
On the other hand, $$c_2(Q) = c_2(\Omega^3_X) = c_2(X) + 3K_X^2,$$
in total  
$$  c_2(X) = -K_X^2. \eqno (**) $$
By L\"ubke's inequality the vanishing $c_2(2K_X \otimes Q^*) = 0$ implies 
$$K_X^2 \cdot \omega^2 = c_1(2 K_X \otimes Q^*)^2 \cdot \omega^2 \leq 4c_2(2 K_X \otimes Q^*) \cdot \omega^2 = 0.$$
By Lemma \ref{lquadomega} above one has $K_X^2 \cdot \omega^2 \geq 0.$ 
Thus equality holds in L\"ubke's inequality
and $Q^*$ is locally free and (hermitian) projectively flat \cite[Cor.3]{BS94}. 
Furthermore $Z = \emptyset,$ since a section in a locally
free sheaf of rank two cannot vanish in codimension three. 

We consider the hermitian flat vector bundle
$$ V = Q \otimes Q^* $$
and claim that there is no finite \'etale cover $f: X' \to X$ such that $f^*(V) $ is trivial. Indeed, the surjection $2K_X \otimes Q^* \lra K_X$ yields $Q \otimes Q^* \lra -K_X \otimes Q$. If $f$ exists, then $f^*(-K_X \otimes Q)$ is spanned and so is $\det f^*(-K_X \otimes Q) = K_{X'}$, contradicting our assumption.

Therefore $V$
defines a unitary representation
$$ \rho: \pi_1(X) \to U(4) $$
with infinite image. 
Hence by \cite[Prop.2.4.2]{Mok00} some finite \'etale cover of $X$ has a map onto a torus, contradicting $\tilde q(X) = 0$, or a meromorphic map to a projective manifold of general type,
contradicting $a(X) = 0$. This completes the proof.  
\end{proof}

\subsection{A deformation argument}

In the preceding section we proved $I_4$ for simple fourfolds that do not contain surfaces
or divisors. A posteriori the additional condition is always satisfied since a simple 
K\"ahler manifold that is torus has no positive-dimensional subvarieties.
Since it seems quite hard to show the absence of these subvarieties,
we develop in this section a deformation-theoretic approach which replaces
in some cases the extension theorem used in the projective case.

\begin{lemma} \label{lemmathreefold}
Let $X$ be a normal K\"ahler Gorenstein threefold with only canonical singularities. Assume that $K_X \equiv 0$ and that the universal covering of $X$ is not covered by positive dimensional analytic subsets. Then there exists a finite \'etale cover $X' \to X$ such that $X'$ is a torus.
\end{lemma} 

\begin{proof} Let $\pi: \hat X \to X$ be a desingularisation.
Then $\kappa(\hat X)=0$, so by \cite{CZ05} the manifold $\hat X$  has a finite \'etale cover $X' \to \hat X$ such that $X'$ is bimeromorphic to a torus $T$.
In particular $\pi_1(\hat X)$ is almost abelian. Since $X$ is normal and $\pi$ birational,
it follows that $\pi_1(X)$ is almost abelian. Thus there exists a finite \'etale cover $X' \to X$ with
$\pi_1(X')$ abelian of rank $6$, one sees easily that $X'$ is isomorphic to its Albanese torus. 
For the algebraic case see also \cite[Cor.8.4]{Kaw85}. 
\end{proof} 

\begin{lemma} \label{lemmacanonicalsystem}
Let $X$ be a compact K\"ahler fourfold without rational curves.
Suppose that $\kappa(X) \geq 0$ and there exists a K\"ahler form $\omega$
such that $K_X^2 \cdot \omega^2 \leq 0$. 

Then we have $K_X^2=0$ and $K_X$ is nef. 
Let $D$ be an irreducible component of an effective divisor in $|m K_X|$. Then $D$
is a connected component of the canonical divisor, it has canonical singularities and $K_D \equiv 0$. 
\end{lemma}

\begin{remarks} \label{remarkequiv}

1) The statement generalises to arbitrary dimension $n$ if one assumes
that a compact K\"ahler manifold of dimension $n-1$ has a pseudoeffective
canonical divisor if and only if it is not uniruled. For projective manifolds this
is due to \cite{BDPP04}, for compact K\"ahler threefolds to \cite{Bru06}.

2) The lemma shows that in our case the condition $K_X^2 \cdot \omega^2 \leq 0$ is equivalent
to $K_X^2=0$. Since a-priori it is not clear whether $K_X$ is nef, this is not obvious.
\end{remarks}

\begin{proof}
Let $\sum a_i D_i$ be an effective divisor in some pluricanonical system $|m K_X|$.
Since $K_X^2 \cdot \omega^2 \leq 0$, there exists an irreducible component, say $D_1$
such that $K_X \cdot D_1 \cdot \omega^2 \leq 0$.
Denote by \holom{\nu}{\bar D_1}{D_1} the normalisation. 
Then one deduces easily from the adjunction formula and \cite{Rei94} that 
$$
K_{\bar D_1} \sim_\Q \nu^* (\frac{m}{a_1}+1) K_X|_{D_1} - \nu^* \left( \sum_{i \geq 2} \frac{a_i}{a_1} (D_i \cap D_1) \right) - N 
$$
where $N$ is an effective Weil divisor defined by the conductor. In particular one has
$$
K_{\bar D_1} \cdot \nu^* \omega^2|_{D_1}  
=
\left[
(\frac{m}{a_1}+1) K_X \cdot D_1 
- \left( \sum_{i \geq 2} \frac{a_i}{a_1} (D_i \cap D_1) \right) - N 
\right] \cdot \omega^2 \leq 0
$$
and the equality is strict if one of the divisors $D_i \cap D_1$ or $N$ is non-empty or $K_X \cdot D_1 \cdot \omega^2 < 0$. Yet in this case if we take a desingularisation \holom{\tau}{\hat{D_1}}{\bar D_1}, we get
$K_{\hat{D_1}} \cdot \tau^* \nu^* \omega^2|_{D_1}<0$. In particular $K_{\hat{D_1}}$ is not pseudoeffective,
hence $\hat{D_1}$ is uniruled \cite{Bru06}. Thus $X$ contains rational curves, a contradiction.
Hence $D_1$ is normal and a connected component of $\supp \sum a_i D_i$.
In particular we have $D_1|_{D_1} \equiv \frac{m}{a_1} K_X|_{D_1}$, so
$$  
K_{D_1} \cdot \omega^2|_{D_1} = (\frac{m}{a_1}+1) K_X \cdot D_1 \cdot \omega^2 =0
$$ 
implies that $K_{D_1}$ is numerically trivial.  

Let $\holom{\tau}{\hat{D_1}}{D_1}$ be a resolution of singularities, then we have 
$$
K_{\hat{D_1}} \equiv \sum c_j E_j,
$$
where the $E_j$ are exceptional divisors. Since $\hat{D_1}$ is not uniruled, the divisor $K_{\hat{D_1}}$
is pseudoeffective. By Boucksom's Zariski decomposition \cite{Bou04} one has 
$K_{\hat{D_1}} \equiv \sum \nu_l F_l+L$ with $\nu_l \geq 0$ and $F_l$ effective $\R$-divisors
and $L$ an $\R$-divisor class that is nef in codimension one. Since
$$
0 \equiv \tau_* K_{\hat{D_1}} \equiv \tau_* (\sum \nu_l F_l) + \tau_* L,
$$ 
we see that all the divisors $F_l$ are $\tau$-exceptional and $L \equiv 0$ by a version of the negativity lemma
(a linear combination of components of the exceptional locus of $\tau$ can never be nef in codimension 1). 
Since the exceptional divisors are linearly independent in the N\'eron-Severi group we obtain 
that for every $j$ there exists a $l$ such that $c_j=\nu_l$. Thus all the discrepancies are non-negative
and $D_1$ has canonical singularities.

Note finally that $K_X \cdot D_1 \cdot \omega^2=0$ 
and $K_X^2 \cdot \omega^2 \leq 0$ implies that $K_X \cdot (\sum_{i \geq 2} a_i D_i) \cdot \omega^2 \leq 0$,
so we conclude by induction that the statement holds for every irreducible component.
\end{proof}

\begin{proposition} \label{propositionabundance}
Let $X$ be a smooth compact K\"ahler fourfold 
such that $\tilde X$ is Stein.  
Suppose that $K_X \not \equiv 0,$ that $\kappa(X) \geq 0$ and $K_X^2 \cdot \omega^2 \leq 0$ for some K\"ahler form $\omega$. Then we have $K_X \equiv 0$ or $\kappa(X)=1$.
\end{proposition}

\begin{proof}
Suppose that $K_X \not\equiv 0$. Since $K_X^2=0$ (cf. Remark \ref{remarkequiv}), we have
$\kappa(X) \leq 1$. If $q(X)>0$ we note that the abundance conjecture holds for compact K\"ahler manifolds
of dimension $\leq 3$ with generically large fundamental group \cite{CZ05}.
We can then use a $C_{n,m}$-argument to conclude, cf. the proof of Proposition \ref{propositionpseudoeffective}.

Suppose now that $q(X)=0$. 
Let $\sum a_i D_i$ be an effective divisor in some pluricanonical system $|m K_X|$,
and set $D := D_1$. For every $n \in \N$ the exact sequence
\[
0 \rightarrow \sO_X \rightarrow \sO_X(nD) \rightarrow \sO_{nD}(nD) \rightarrow 0
\]
induces a short exact sequence
\[
0 \rightarrow H^0(X, \sO_X) \rightarrow H^0(X, \sO_X(nD)) \rightarrow H^0(nD, \sO_{nD}(nD)) \rightarrow H^1(X, \sO_X)=0.
\]
We claim that $H^0(nD, \sO_{nD}(nD)) \neq 0$ for some $n$; then we get  $h^0(X, \sO_X(nD)) \geq 2$
and hence $\kappa(X) \geq 1$. 

{\em Proof of the claim.}
By the Lemmas \ref{lemmathreefold} and \ref{lemmacanonicalsystem} 
we know that $D$ is an \'etale quotient of a torus,
in particular it is smooth. Since $D$ does not meet
the irreducible components $D_i$ for $i \neq 1$, we have
\[
K_D \sim (K_X+D)|_D \sim_\Q  (\frac{a_1}{m}+1) D|_D.
\]
In particular since $K_D$ is a torsion line bundle, we see that $D|_D$ is a torsion line bundle.
Let now $D \subset U$ be an analytic neighbourhood, then $D|_U$ is torsion, so there exists
a cyclic \'etale covering \holom{g}{N}{U} such that $g^* D|_U$ is trivial. In particular
the normal bundle of $S:=g^*D$ is trivial. Similarly we see that $K_N$ is torsion, so up to taking
another cyclic covering we can suppose that $K_N|_S$ is trivial. In particular the canonical bundle
of $S$ is trivial by adjunction, so the versal deformation
space of $S$ is smooth by \cite[Cor.2]{ Ran92}. Thus $S$ satisfies the conditions of \cite[Thm.4.2]{Miy88},
which by \cite[Cor.4.6]{Miy88} implies that
\[
h^0(nD, \sO_{nD}(nD))
\]
grows with order $n$. In particular it is non-zero for some $n$ sufficiently large and divisible. 
\end{proof}

By \cite{Nak99b} the proposition has the following

\begin{corollary} \label{corollaryfoure}
Let $X$ be a smooth compact K\"ahler fourfold 
such that $\kappa(X) \geq 0$ and $K_X^2 \cdot \omega^2 \leq 0$ for some K\"ahler form $\omega$. Then $I_4$ holds for $X$. 
\end{corollary}

\section{Non-vanishing via the Hard Lefschetz theorem}

\begin{theorem} \label{theoremcohomology}
Let $X$ be a smooth projective fourfold 
such that $\tilde X \simeq \bC^4$ or, more generally, that there is a proper modification $\tilde \C^4 \to \C^4$ and a surjective map $\tilde \C^4 \to \tilde X.$ 
Suppose that $\kappa (X) = - \infty.$ 
Then there exists $C > 0$ such that for all $m \in \bN:$
$$ h^0(X,\Omega^q_X(mK_X)) \leq C.$$
\end{theorem} 

\begin{proof} 
Suppose to the contrary that there is some $q$ (necessarily $1 \leq q \leq 3$) such that $h^0(X,\Omega^q_X(mK_X))$ is 
not bounded. Then we find a a constant $K > 0$ such that 
$$ h^0(X,\Omega^q_X(mK_X)) \geq Km$$
for $m \gg 0.$ Let $\sS_m \subset \Omega^q_X(mK_X) $ be the subsheaf generated by the global sections of $\Omega^q_X(mK_X).$ 
Let $r_m$ be the rank of $\sS_m$. If $r_m= \rk \Omega^q_X,$ then $\Omega^q_X(mK_X)$ is generically spanned, so does
its determinant, and hence $\kappa (X) \geq 0.$ \\
So $\sS_m$ is always a subsheaf of smaller rank. We set 
$$ \sL_m := \det \sS_m \subset  \bigwedge^{r_m} \Omega^q_X(mK_X). $$
We fix $m_0$ such that $h^0(X, \sL_{m_0}) \geq 2$ and set $\sL = \sL_{m_0}; r = r_{m_0}.$ 
We obtain an exact sequence
$$ 0 \to \sL \to \bigwedge^r \Omega^q(m_0K_X) \to Q \to 0, $$
with quotient $Q$ which we can suppose to be torsion-free. 
By \cite[Thm.1.4]{CP11} we know that $\det  Q$ is pseudo-effective, and for a suitable number $N,$ there is a decomposition
$$ NK_X = \sL + \det Q. $$ 
Applying \cite[Thm.2.3,Cor.2.10]{CP11} it follows that $\kappa (\sL) \leq 2.$ 
Let 
$$ \phi: X \dasharrow Y$$
be the rational map defined by $H^0(X,\sL),$ so that $1 \leq \dim Y \leq 2.$ 
Then a desingularization of a general fibre $X_y$ is not of general type by \cite[Prop.2.9]{CP11}. 
Hence Lemma \ref{lemmainduction} applies, and some finite \'etale cover of $X$ is bimeromorphically a torus, contradicting the assumption $\kappa (X) = - \infty.$ 
\end{proof} 

\begin{corollary} \label{corollarymultiplier}
Let $X$ be a smooth projective fourfold with $\tilde X \simeq \bC^4.$ Let $h$ be a possibly singular
metric on $K_X$ with semi-positive curvature current ($h$ exists since $K_X$ is nef). Let $\sJ(h^m) $ be the 
multiplier ideal of $h_m$ and let $V_m$ be the subscheme defined by $\sJ(h^m).$ If $\dim V_m \leq 1$ for all $m \gg 0$, then
$K_X^2 \cdot c_2(X) = 0.$ 
\end{corollary} 

\begin{proof} We follow the line of arguments in \cite{DPS01} and \cite{COP10}. 
Assume that $K_X^2 \cdot c_2(X) \ne 0.$ Then, due to Miyaoka's inequality $K_X^2 \leq 3c_2(X), $ we have
$K_X^2 \cdot c_2(X) \geq K_X^4 \geq 0,$
hence 
$$ K_X^2 \cdot c_2(X) > 0.$$ 
Applying Riemann-Roch and having in mind that $K_X^4 = 0$, there is a constant $C > 0,$ such that 
$$ h^2(X,mK_X) \geq Cm^2.$$
Since $H^2(V_m, K_X \otimes mK_X \vert _{V_m})$ vanishes for reasons of dimension,
the cohomology sequence
$$ H^2(X,K_X \otimes mK_X \otimes \sJ(h^m)) \to H^2(X,K_X \otimes mK_X) \to H^2(V_m, K_X \otimes mK_X \vert _{V_m}) = 0,$$
implies that
$$ h^2(K_X \otimes mK_X \otimes \sJ(h^m)) \geq Cm^2. $$
By the Hard Lefschetz Theorem \cite[Thm.2.1.1]{DPS01} this implies 
$$ h^0(X,\Omega^2_X \otimes mK_X \otimes \sJ(h^m)) \geq Cm^2,$$
contradicting the previous theorem. 
\end{proof} 

\begin{corollary}  \label{corollaryfourf}
Let $X$ be a compact K\"ahler fourfold with $\tilde X \simeq \bC^4.$ 
If $K_X$ is hermitian semi-positive and $K_X^2 \cdot c_2(X) \neq 0$,
then $\kappa(X) \geq 0$.
\end{corollary} 

\begin{proof} In the proof of Corollary \ref{corollarymultiplier} the projectivity assumption
is used in the proof of the previous corollary only in two places. \\

1) We used Miyaoka's inequality, which is unknown in the K\"ahler case. If however $K_X^2 \cdot c_2(X) < 0,$ then 
we obtain at least quadratic growth of $H^1(X,K_X \otimes mK_X)$ or  $H^3(X,K_X \otimes mK_X) = 0$ and conclude in the same way as before, using our
stronger assumptions. Notice however that  $H^3(X,K_X \otimes mK_X) = 0$ due to \cite[Thm.0.1]{DP03}. \\

2) We used Theorem \ref{theoremcohomology} to produce a contradiction. In the K\"ahler setting, we may a priori assume that $a(X) = 0$ due to Proposition \ref{propositionmiddlecases}.
But this contradicts the growth condition
$$ h^0(X,\Omega^q_X(mK_X)) \geq Cm^2. $$ 
\end{proof}


\begin{thebibliography}{BCHM10}

\bibitem[BCHM10]{BCHM10}
Caucher Birkar, Paolo Cascini, Christopher~D. Hacon, and James McKernan.
\newblock Existence of minimal models for varieties of log general type.
\newblock {\em J. Amer. Math. Soc.}, 23(2):405--468, 2010.

\bibitem[BDPP04]{BDPP04}
S{\'e}bastien Boucksom, Jean-Pierre Demailly, Mihai P{\u a}un, and Thomas
  Peternell.
\newblock The pseudo-effective cone of a compact {K\"a}hler manifold and
  varieties of negative {K}odaira dimension.
\newblock {\em arxiv preprint, to appear in J.A.G.}, 0405285, 2004.

\bibitem[Bou04]{Bou04}
S{\'e}bastien Boucksom.
\newblock Divisorial {Z}ariski decompositions on compact complex manifolds.
\newblock {\em Ann. Sci. \'Ecole Norm. Sup. (4)}, 37(1):45--76, 2004.

\bibitem[Bru06]{Bru06}
Marco Brunella.
\newblock A positivity property for foliations on compact {K}\"ahler manifolds.
\newblock {\em Internat. J. Math.}, 17(1):35--43, 2006.

\bibitem[BS94]{BS94}
Shigetoshi Bando and Yum-Tong Siu.
\newblock Stable sheaves and {E}instein-{H}ermitian metrics.
\newblock In {\em Geometry and analysis on complex manifolds}, pages 39--50.
  World Sci. Publ., River Edge, NJ, 1994.

\bibitem[Cam81]{Cam81}
Fr{\'e}d{\'e}ric Campana.
\newblock Cor\'eduction alg\'ebrique d'un espace analytique faiblement
  k\"ahl\'erien compact.
\newblock {\em Invent. Math.}, 63(2):187--223, 1981.

\bibitem[Cam85]{Cam85b}
Fr{\'e}d{\'e}ric Campana.
\newblock R\'eduction d'{A}lban\`ese d'un morphisme propre et faiblement
  k\"ahl\'erien. {II}. {G}roupes d'automorphismes relatifs.
\newblock {\em Compositio Math.}, 54(3):399--416, 1985.

\bibitem[Cam95]{Cam95}
Fr{\'e}d{\'e}ric Campana.
\newblock Fundamental group and positivity of cotangent bundles of compact
  {K}\"ahler manifolds.
\newblock {\em J. Algebraic Geom.}, 4(3):487--502, 1995.

\bibitem[Cam98]{Cam98b}
Fr{\'e}d{\'e}ric Campana.
\newblock Connexit\'e ab\'elienne des vari\'et\'es k\"ahl\'eriennes compactes.
\newblock {\em Bull. Soc. Math. France}, 126(4):483--506, 1998.

\bibitem[Cam04]{Cam04a}
Fr{\'e}d{\'e}ric Campana.
\newblock Orbifolds, special varieties and classification theory: an appendix.
\newblock {\em Ann. Inst. Fourier (Grenoble)}, 54(3):631--665, 2004.

\bibitem[Cla10]{Cla10}
Beno{\^{\i}}t Claudon.
\newblock Invariance de la {$\Gamma$}-dimension pour certaines familles
  k\"ahl\'eriennes de dimension 3.
\newblock {\em Math. Z.}, 266(2):265--284, 2010.

\bibitem[COP10]{COP10}
Fr{\'e}d{\'e}ric Campana, Keiji Oguiso, and Thomas Peternell.
\newblock Non-algebraic hyperk\"ahler manifolds.
\newblock {\em J. Differential Geom.}, 85(3):397--424, 2010.

\bibitem[CP11]{CP11}
Fr{\'e}d{\'e}ric Campana and Thomas Peternell.
\newblock Geometric stability of the cotangent bundle and the universal cover
  of a projective manifold.
\newblock {\em Bulletin de la S.M.F.}, 139(1):41--74, 2011.

\bibitem[CZ05]{CZ05}
Fr{\'e}d{\'e}ric Campana and Qi~Zhang.
\newblock Compact {K}\"ahler threefolds of {$\pi_1$}-general type.
\newblock In {\em Recent progress in arithmetic and algebraic geometry}, volume
  386 of {\em Contemp. Math.}, pages 1--12. Amer. Math. Soc., Providence, RI,
  2005.

\bibitem[Dem92]{Dem92}
Jean-Pierre Demailly.
\newblock Regularization of closed positive currents and intersection theory.
\newblock {\em J. Algebraic Geom.}, 1(3):361--409, 1992.

\bibitem[DHP10]{DHP10}
Jean-Pierre Demailly, Christopher~D. Hacon, and Mihai P{\v a}un.
\newblock Extension theorems, non-vanishing and the existence of good minimal
  models.
\newblock {\em arxiv}, 1012.0493, 2010.

\bibitem[DP03]{DP03}
Jean-Pierre Demailly and Thomas Peternell.
\newblock A {K}awamata-{V}iehweg vanishing theorem on compact {K}\"ahler
  manifolds.
\newblock {\em J. Differential Geom.}, 63(2):231--277, 2003.

\bibitem[DPS01]{DPS01}
Jean-Pierre Demailly, Thomas Peternell, and Michael Schneider.
\newblock Pseudo-effective line bundles on compact {K}\"ahler manifolds.
\newblock {\em Internat. J. Math.}, 12(6):689--741, 2001.

\bibitem[FF79]{FF79}
Gerd Fischer and Otto Forster.
\newblock Ein {E}ndlichkeitssatz f\"ur {H}yperfl\"achen auf kompakten komplexen
  {R}\"aumen.
\newblock {\em J. Reine Angew. Math.}, 306:88--93, 1979.

\bibitem[Fuj78]{Fuj78}
Takao Fujita.
\newblock On {K\"a}hler fiber spaces over curves.
\newblock {\em J. Math. Soc. Japan}, pages 779--794, 1978.

\bibitem[Fuj83]{Fuj83}
Akira Fujiki.
\newblock On the structure of compact complex manifolds in {${\cal C}$}.
\newblock In {\em Algebraic varieties and analytic varieties ({T}okyo, 1981)},
  volume~1 of {\em Adv. Stud. Pure Math.}, pages 231--302. North-Holland,
  Amsterdam, 1983.

\bibitem[Fuj07]{Fuj07}
Osamu Fujino.
\newblock Special termination and reduction to pl flips.
\newblock In {\em Flips for 3-folds and 4-folds}, volume~35 of {\em Oxford
  Lecture Ser. Math. Appl.}, pages 63--75. Oxford Univ. Press, Oxford, 2007.

\bibitem[Kaw85]{Kaw85}
Yujiro Kawamata.
\newblock Minimal models and the {K}odaira dimension of algebraic fiber spaces.
\newblock {\em J. Reine Angew. Math.}, 363:1--46, 1985.

\bibitem[KM98]{KM98}
J{\'a}nos Koll{\'a}r and Shigefumi Mori.
\newblock {\em Birational geometry of algebraic varieties}, volume 134 of {\em
  Cambridge Tracts in Mathematics}.
\newblock Cambridge University Press, Cambridge, 1998.

\bibitem[KO75]{KO75}
Shoshichi Kobayashi and Takushiro Ochiai.
\newblock Meromorphic mappings onto compact complex spaces of general type.
\newblock {\em Invent. Math.}, 31(1):7--16, 1975.

\bibitem[Kol93]{Kol93}
J{\'a}nos Koll{\'a}r.
\newblock Shafarevich maps and plurigenera of algebraic varieties.
\newblock {\em Invent. Math.}, 113(1):177--215, 1993.

\bibitem[Kwc92]{Kol92}
J{\'a}nos Koll{\'a}r~(with 14~coauthors).
\newblock {\em Flips and abundance for algebraic threefolds}.
\newblock Soci\'et\'e Math\'ematique de France, Paris, 1992.
\newblock Ast{\'e}risque No. 211 (1992).

\bibitem[Lai10]{Lai10}
Ching-Jui Lai.
\newblock Varieties fibered by good minimal models.
\newblock {\em Mathematische Annalen}, 2010.

\bibitem[Miy88]{Miy88}
Yoichi Miyaoka.
\newblock Abundance conjecture for {$3$}-folds: case {$\nu=1$}.
\newblock {\em Compositio Math.}, 68(2):203--220, 1988.

\bibitem[Mok00]{Mok00}
Ngaiming Mok.
\newblock Fibrations of compact {K}\"ahler manifolds in terms of cohomological
  properties of their fundamental groups.
\newblock {\em Ann. Inst. Fourier (Grenoble)}, 50(2):633--675, 2000.

\bibitem[Nak99a]{Nak99b}
Noboru Nakayama.
\newblock Compact {K\"a}hler manifolds whose universal covering spaces are
  biholomorphic to {${\bf C}^n$}.
\newblock {\em RIMS preprint}, 1230, 1999.

\bibitem[Nak99b]{Nak99}
Noboru Nakayama.
\newblock Projective algebraic varieties whose universal covering spaces are
  biholomorphic to {${\bf C}^n$}.
\newblock {\em J. Math. Soc. Japan}, 51(3):643--654, 1999.

\bibitem[Ran92]{Ran92}
Ziv Ran.
\newblock Deformations of manifolds with torsion or negative canonical bundle.
\newblock {\em J. Algebraic Geom.}, 1(2):279--291, 1992.

\bibitem[Rei94]{Rei94}
Miles Reid.
\newblock Nonnormal del {P}ezzo surfaces.
\newblock {\em Publ. Res. Inst. Math. Sci.}, 30(5):695--727, 1994.

\bibitem[Uen87]{Uen87}
Kenji Ueno.
\newblock On compact analytic threefolds with nontrivial {A}lbanese tori.
\newblock {\em Math. Ann.}, 278(1-4):41--70, 1987.

\end{thebibliography}
\end{document}